\documentclass[12pt, a4paper]{article}

\usepackage{latexsym}
\usepackage{graphicx, float}
\usepackage{amssymb,amsfonts,amsthm,amsmath,graphicx,url}

\long\def\delete#1{}

\newtheorem{theorem}{Theorem}
\newtheorem{lemma}[theorem]{Lemma}
\newtheorem{corollary}[theorem]{Corollary}
\newtheorem{definition}[theorem]{Definition}

\input amssym.def
\input amssym.tex

\newcommand{\bmat}[1]{\begin{bmatrix}#1\end{bmatrix}}
\newcommand{\dmat}[1]{\begin{vmatrix}#1\end{vmatrix}}

\newcommand{\be}{\begin{equation}}
\newcommand{\ee}{\end{equation}}
\newcommand{\bea}{\begin{eqnarray}}
\newcommand{\eea}{\end{eqnarray}}
\newcommand{\bean}{\begin{eqnarray*}}
\newcommand{\eean}{\end{eqnarray*}}

\def\qed{\hfill$\Box$\vspace{12pt}}

\def\la{\langle}
\def\ra{\rangle}

\def\FFF{\Bbb F}

\def\FFF{\mathbb{F}}

\def\BB{{\cal B}}

\def\PP{{\cal P}}

\def\b0{{\bf 0}}

\def\be{{\bf e}}

\def\bu{{\bf u}}
\def\bv{{\bf v}}

\def\De{\Delta}
\def\Ga{\Gamma}

\def\b{\beta}

\def\l{\lambda}
\def\s{\sigma}
\def\t{\tau}

\def\Aut{{\rm Aut}}

\def\PSU{{\rm PSU}}
\def\PGU{{\rm PGU}}

\def\GU{{\rm GU}}
\def\SU{{\rm SU}}

\def\PG{{\rm PG}}

\def\PGammaU{{\rm P\Gamma U}}

\title{Unitary graphs}
\author{Sanming Zhou\\
Department of Mathematics and Statistics\\
The University of Melbourne, VIC 3010, Australia\\
Email: smzhou@ms.unimelb.edu.au \\}

\begin{document}
\openup 0.8\jot
\maketitle


\begin{abstract}
Unitary graphs are arc-transitive graphs with vertices the flags of Hermitian unitals and edges defined by certain elements of the underlying finite fields. They played a significant role in a recent classification of a class of arc-transitive graphs that admit an automorphism group acting imprimitively on the vertices. In this paper we prove that all unitary graphs are connected of diameter two and girth three. Based on this we obtain, for any prime power $q > 2$, a lower bound of order $O(\De^{5/3})$ on the maximum number of vertices in an arc-transitive graph of degree $\De = q(q^2-1)$ and diameter two.  

\end{abstract}

\section{Introduction}

We study a family of arc-transitive graphs \cite{Biggs} associated with Hermitian unitals. Such graphs are called unitary graphs \cite{GMPZ} due to their connections with unitary groups of degree three over a Galois field. The vertices of a unitary graph are the flags of a Hermitian unital, and the adjacency relation is determined by two linear equations defining the line-components of the flags involved. Unitary graphs played an important role in a recent classification \cite{GMPZ} of a class of arc-transitive graphs that admit an automorphism group acting imprimitively on the vertices. (A graph is {\em arc-transitive} if its automorphism group is transitive on the set of ordered pairs of adjacent vertices.) With focus on combinatorial aspects of unitary graphs, in the present paper we prove that all unitary graphs are connected with large order (compared with their degrees), small diameter and small girth. Based on this we then obtain, for any prime power $q > 2$, a lower bound on the maximum order (number of vertices) of an arc-transitive graph of degree $q(q^2 - 1)$ and diameter two.  

The distance between two vertices in a graph is the length of a shortest path joining them, and $\infty$ if there is no path between the two vertices. The \textit{diameter} of a graph is the maximum distance between two vertices in the graph. The \textit{girth} of a graph is the length of a shortest cycle, and $\infty$ if the graph contains no cycle at all. Two vertices are neighbours of each other if they are adjacent in the graph. 

Denote by
\begin{equation}
\label{eq:psi}
\psi: x \mapsto x^p,\; x \in  \FFF_{q^2}
\end{equation}
the Frobenius map for the Galois field $\FFF_{q^2}$, where $p$ is a prime and $q > 2$ is a power of $p$. We postpone the definition of the unitary graph $\Ga_{r, \l}(q)$ and the $\PGU(3,q) \rtimes \la \psi^r \ra$-invariant partition $\BB$ of its vertex set to the next section (see Definition \ref{def:ugraph} and (\ref{eq:B}) respectively). The following is the first main result of this paper.

\begin{theorem}
\label{thm:con}
Let $q = p^e > 2$ be a prime power and $r \ge 1$ a divisor of $2e$. Let $\l \in \FFF_{q^2}^{*}$ be such that $\l^q$ belongs to the $\la \psi^r \ra$-orbit on $\FFF_{q^2}$ containing $\l$, and let $k = k_{r,\l}(q)$ denote the size of this $\la \psi^r \ra$-orbit. Then the unitary graph $\Ga_{r, \l}(q)$ is connected of diameter two and girth three. Moreover, the following hold for $\Ga_{r, \l}(q)$:
\begin{itemize}
\item[\rm (a)] any two vertices in different blocks of $\BB$ have at least $q^2 (q-2)$ common neighbours;
\item[\rm (b)] any two vertices in the same block of $\BB$ have exactly $k(k-1)q$ common neighbours. 
\end{itemize}
\end{theorem}
 
Given integers $\De, D \ge 1$, the well known degree-diameter problem \cite{MS} asks for finding the maximum order $N_{\De, D}$ of a graph of maximum degree $\De$ and diameter at most $D$ together with the corresponding extremal graphs. Denote by $N^{at}(\De, D)$ the maximum order of an arc-transitive graph of degree $\De$ and diameter at most $D$. Based on Theorem \ref{thm:con} we obtain the following lower bound on $N^{at}(\De, 2)$. 

\begin{theorem}
\label{thm:deg diam}
For any prime power $q > 2$, 
\begin{equation}
\label{eq:deg diam}
N^{at}(q(q^2 - 1), 2) \ge q^2 (q^3 + 1).
\end{equation}
In particular, for $\De = q(q^2 - 1)$,  
\begin{equation}
\label{eq:deg diam1}
N^{at}(\De, 2) \ge \De^{5/3} + \De + \De^{2/3} + \De^{1/3}.
\end{equation}
\end{theorem}

As far as we know, these bounds are the first general lower bounds for the arc-transitive version of the degree-diameter problem, despite the fact that a huge amount of work has been done \cite{MS} on this problem for general graphs and its restrictions to several other graph classes (e.g. bipartite graphs, vertex-transitive graphs, Cayley graphs). The reader may compare (\ref{eq:deg diam1}) with the well known Moore bound $N_{\De, 2} \le \De^2 + 1$ (for general graphs) and consult \cite{MS} for the state-of-the-art of the degree-diameter problem.  

The extremal graphs that prove (\ref{eq:deg diam}) form a subfamily of the family of unitary graphs as we will see in the proof of Theorem \ref{thm:deg diam}. The smallest unitary graphs arise when $q=p=3$, and in this case (\ref{eq:deg diam}) gives $N^{at}(24, 2) \ge 3^2 (3^3 + 1) = 252$. Our graphs are constructed from Hermitian unitals, which are well-known doubly point-transitive linear spaces. In this regard we would like to mention that some efforts have been made to construct graphs using certain finite geometries that give good bounds for the vertex-transitive version of the degree-diameter problem; see \cite{ANS, KKKRS} for example.

We will give the definition of the unitary graph $\Ga_{r, \l}(q)$ and related concepts in the next section. The proof of Theorems \ref{thm:con} and \ref{thm:deg diam} together with some preparatory results will be given in Section \ref{sec:main}. We conclude the paper with remarks on Theorem \ref{thm:deg diam} and related questions on the order of $N^{at}(\De, 2)$.

\section{Unitary graphs}

In order to make this paper reasonably self-contained, we first gather basic definitions and results on unitary groups and Hermitian unitals. After this we will give the definition of a unitary graph. The reader is referred to \cite{Dixon-Mortimer, HP, O'Nan, Taylor} for more information on unitary groups and Hermitian unitals, and to \cite{Dixon-Mortimer} for undefined terminology on permutation groups.  

Let $q = p^e > 2$ with $p$ a prime. The mapping $\s: x \mapsto x^q$ is an automorphism of the Galois field $\mathbb{F}_{q^2}$. The Galois field $\mathbb{F}_{q}$ is then the fixed field of this automorphism. Let $V(3, q^2)$ be a 3-dimensional vector space over $\mathbb{F}_{q^2}$ and $\b: V(3, q^2) \times V(3, q^2) \rightarrow \mathbb{F}_{q^2}$ a nondegenerate $\s$-Hermitian form (that is, $\b$ is sesquilinear such that $\b(a \bu, b \bv) = a b^q \b(\bu, \bv)$ and $\b(\bu, \bv) = \b(\bv, \bu)^q$). The full unitary group $\Gamma U(3, q)$ consists of those semilinear transformations of $V(3, q^2)$ that induce a collineation of $\PG(2, q^2)$ which commutes with $\b$. The general unitary group $\GU(3, q) = \Gamma U(3, q) \cap GL(3, q^2)$ is the group of nonsingular linear transformations of $V(3, q^2)$ leaving $\b$ invariant. The projective unitary group $\PGU(3, q)$ is the quotient group $\GU(3, q)/Z$, where $Z = \{a I: a \in \mathbb{F}_{q^2}, a^{q+1} = 1\}$ is the center of $\GU(3,q)$ and $I$ the identity transformation. The special projective unitary group $\PSU(3, q)$ is the quotient group $\SU(3, q)Z/Z$, where $\SU(3, q)$ is the subgroup of $\GU(3, q)$ consisting of linear transformations of unit determinant. $\PSU(3, q)$ is equal to $\PGU(3, q)$ if $3$ is not a divisor of $q+1$, and is a subgroup of $\PGU(3, q)$ of index 3 otherwise. It is well known that the automorphism group of $\PSU(3, q)$ is given by the semi-direct product $\PGammaU(3, q) := \PGU(3, q) \rtimes \la \psi \ra$, where $\psi$ is the Frobenius map as defined in (\ref{eq:psi}).

Choosing an appropriate basis for $V(3, q^2)$ allows us to identify vectors of $V(3, q^2)$ with their coordinates and express the corresponding Hermitian matrix of $\b$ by 
$$
D = \bmat{-1 & 0 & 0\\0 & 0 & 1\\0 & 1 & 0}. 
$$
Thus, for $\bu_1 = (x_1, y_1, z_1), \bu_2 = (x_2, y_2, z_2) \in V(3, q^2)$,
$$
\b(\bu_1, \bu_2) = -x_1 x_2^q + y_1 z_2^q + z_1 y_2^q.
$$ 
If $\b(\bu_1, \bu_2) = 0$, then $\bu_1$ and $\bu_2$ are called {\em orthogonal} (with respect to $\b$). A vector $\bu = (x, y, z) \in V(3, q^2)$ is called {\em isotropic} if it is orthogonal to itself, that is, $x^{q+1} = y z^q + z y^q$, and {\em nonisotropic} otherwise. Let 
$$
X = \{\la x, y, z \ra: x, y, z \in \FFF_{q^2}, x^{q+1} = y z^q + z y^q\}
$$
be the set of 1-dimensional subspaces of $V(3, q^2)$ spanned by its isotropic vectors. Hereinafter $\la \bu \ra = \la x, y, z \ra$ denotes the 1-dimensional subspace of $V(3, q^2)$ spanned by $\bu = (x, y, z) \in V(3, q^2)$. 
The elements of $X$ are called the {\em absolute points}. It is well known that $|X| = q^3 + 1$, $\PSU(3, q)$ is 2-transitive on $X$, and $\PGammaU(3, q)$ leaves $X$ invariant. 

If $\bu_1$ and $\bu_2$ are isotropic, then the vector subspace $\la \bu_1, \bu_2 \ra$ of $V(3, q^2)$ spanned by them contains exactly $q+1$ absolute points. The {\em Hermitian unital} $U_{H}(q)$ is the block design \cite{HP} with point set $X$ in which a subset of $X$ is a block (called a {\em line}) precisely when it is the set of absolute points contained in some $\la \bu_1, \bu_2 \ra$. It is well known \cite{O'Nan, Taylor} that $U_{H}(q)$ is a linear space with $q^3+1$ points, $q^2(q^2 - q + 1)$ lines, $q+1$ points in each line, and $q^2$ lines meeting at a point. (A {\em linear space} \cite{Beth-Jung-Lenz} is an incidence structure of points and lines such that any point is incident with at least two lines, any line with at least two points, and any two points are incident with exactly one line.) It was proved in \cite{O'Nan, Taylor} that $\Aut(U_{H}(q)) = \PGammaU(3, q)$. Thus, for every $G$ with $\PSU(3,q) \le G \le \PGammaU(3,q)$, $U_{H}(q)$ is a $G$-doubly point-transitive linear space. This implies that $G$ is also block-transitive and flag-transitive on $U_{H}(q)$, where a {\em flag} is an incident point-line pair.   

A line of $\PG(2, q^2)$ contains either one absolute point or $q+1$ absolute points. In the latter case the set of such $q+1$ absolute points is a line of $U_{H}(q)$; all lines of $U_{H}(q)$ are of this form. So we can represent a line of $U_{H}(q)$ by the homogenous equation of the corresponding line of $\PG(2, q^2)$. 

Denote
$$
V(q) = \mbox{the set of flags of}\ \,U_{H}(q).
$$
\begin{definition}
\label{def:ugraph}
{\em (\cite{GMPZ})~Let $q = p^e > 2$ be a prime power and $r \ge 1$ a divisor of $2e$. Suppose $\l \in \FFF_{q^2}^{*}$ is such that $\l^q$ belongs to the $\la \psi^r \ra$-orbit on $\FFF_{q^2}$ containing $\l$. The {\em unitary graph} $\Ga_{r, \l}(q)$ is defined to be the graph with vertex set $V(q)$ such that $(\la a_1, b_1, c_1 \ra, L_1)$, $(\la a_2, b_2, c_2 \ra, L_2) \in V(q)$ are adjacent if and only if $L_1$ and $L_2$ are given by:
\begin{equation}
\label{eq:l1}
L_1: \dmat{x & a_1 & a_0 + a_2\\y & b_1 & b_0 + b_2\\z & c_1 & c_0 + c_2} = 0
\end{equation}
\begin{equation}
\label{eq:l2}
L_2: \dmat{x & a_2 & a_0+\l^{qp^{ir}} a_1 \\y & b_2 & b_0 + \l^{qp^{ir}}b_1 \\z & c_2 & c_0 + \l^{qp^{ir}}c_1} = 0
\end{equation}
for an integer $0 \le i < 2e/r$ and a nonisotropic $(a_0, b_0, c_0) \in V(3, q^2)$ orthogonal to both $(a_1, b_1, c_1)$ and $(a_2, b_2, c_2)$.}
\end{definition}

The requirement on $\l$ is equivalent to that $\l^{p^{tr}} = \l^q$ for at least one integer $0 \le t < 2e/r$. (But $\Ga_{r, \l}(q)$ is independent of the choice of $t$.) This ensures that $\Ga_{r, \l}(q)$ is well defined as an undirected graph. In fact, since $r$ is a divisor of $2e$, we have $(j+t)r = 2e$ for some integer $j$. Since $\l = \l^{qp^{jr}}$, the equations of $L_1$ and $L_2$ can be rewritten as
$$
L_2: \dmat{x & a_2 & \l a_0+\l^{qp^{ir}+1} a_1 \\y & b_2 & \l b_0 + \l^{qp^{ir}+1}b_1 \\z & c_2 & \l c_0 + \l^{qp^{ir}+1}c_1} = 0,\;\;\;
L_1: \dmat{x & \l^{qp^{ir}+1} a_1 & \l a_0+\l^{qp^{jr}} a_2 \\y & \l^{qp^{ir}+1} b_1 & \l b_0 + \l^{qp^{jr}} b_2 \\z & \l^{qp^{ir}+1} c_1 & \l c_0 + \l^{qp^{jr}} c_2} = 0.
$$
Hence the adjacency relation of $\Ga_{r,\l}(q)$ is symmetric. 

Define 
$$
k_{r,\l}(q) = \frac{|\la \psi^r \ra|}{|\la \psi^r \ra_{\l}|},
$$
where $\la \psi^r \ra_{\l}$ is the stabilizer of $\l$ in $\la \psi^r \ra$. Then $k_{r,\l}(q)$ is the size of the $\la \psi^r \ra$-orbit on $\FFF_{q^2}$ containing $\l$, or the least integer $j \ge 1$ such that $\l^{p^{jr}} = \l$. Of course $k_{r,\l}(q)$ is a divisor of $2e/r$. 

Denote by $B(\s)$ the set of flags of $U_{H}(q)$ with point-entry $\s \in X$. Then
\begin{equation}
\label{eq:B}
\BB = \{B(\s): \s \in X\}
\end{equation}
is a partition of $V(q)$ into $q^3+1$ blocks each with size $q^2$. 

Denote by $L(\s\t)$ the unique line of $U_H(q)$ through two distinct points $\s, \t \in X$.  
Denote
$$
\infty = \la 0, 1, 0 \ra; \quad\quad 0 = \la 0, 0, 1 \ra
$$
$$
L: x = z;\quad\quad N: y=\l^q x;\quad\quad L^*: x = 0.
$$ 
Then $(\infty, L), (0, N) \in V(q)$ and $L^* = L(\infty 0)$.

An {\em arc} of a graph is an ordered pair of adjacent vertices. A graph $\Ga$ is {\em $G$-arc transitive} if $G \le \Aut(\Ga)$ is transitive on the set of vertices of $\Ga$ and also transitive on the set of arcs of $\Ga$. This is to say that any arc of $\Ga$ can be mapped to any other arc of $\Ga$ by an element of $G$, and the same statement holds for vertices. A partition $\PP$ of the vertex set of $\Ga$ is {\em $G$-invariant} if for any block $P \in \PP$ and $g \in G$ the image of $P$ under $g$, $\{\s^g: \s \in P\}$, is a block of $\PP$, where $\s^g$ is the image of $\s$ under $g$. The {\em quotient graph} $\Ga_{\PP}$ is the graph with vertex set $\PP$ such that $P, Q \in \PP$ are adjacent if and only if there is at least one edge of $\Ga$ between $P$ and $Q$. If for any two adjacent $P, Q \in \PP$, all vertices of $P$ except only one have neighbours in $Q$ in the graph $\Ga$, then $\Ga$ is called an {\em almost multicover} \cite{GMPZ} of $\Ga_{\BB}$. (Since $\Ga$ is $G$-arc transitive, if all vertices of $P$ except one have neighbours in $Q$, then all vertices of $Q$ except one have neighbours in $P$, and the subgraph of $\Ga$ induced by $P \cup Q$ with these two exceptional vertices deleted, is a regular bipartite graph.)

Unitary graphs were introduced in \cite{GMPZ} during the classification of a class of imprimitive arc-transitive graphs. A major step towards this classification is the following result which will be used in our proof of Theorem \ref{thm:con}.
\begin{theorem}
\label{thm:ugraph}
(\cite{GMPZ}) $\Ga_{r, \l}(q)$ is a $\PGU(3,q) \rtimes \la \psi^r \ra$-arc transitive graph of degree $kq(q^2-1)$ (where $k = k_{r,\l}(q)$) that admits $\BB$ as a $\PGU(3,q) \rtimes \la \psi^r \ra$-invariant partition such that the quotient graph $\Ga_{r, \l}(q)_{\BB}$ is a complete graph and $\Ga_{r, \l}(q)$ is an almost multicover of $\Ga_{r, \l}(q)_{\BB}$. Moreover, for each pair of distinct points $\s, \t$ of $U_{H}(q)$, $(\s, L(\s \t))$ is the only vertex in $B(\s)$ that has no neighbour in $B(\t)$.
\end{theorem}

\section{Proof of Theorems \ref{thm:con} and \ref{thm:deg diam}}
\label{sec:main}

Throughout this section we denote
$$
\Ga = \Ga_{r, \l}(q); \quad\quad G = \PGU(3,q) \rtimes \la \psi^r \ra; \quad\quad k = k_{r,\l}(q).
$$ 
We need the following two lemmas in the proof of Theorem \ref{thm:con}.

\begin{lemma}
\label{lem:nb}
\begin{itemize}
\item[\rm (a)] $(\la \bu_2 \ra, L_2) \in V(q)$ is adjacent to $(\infty, L)$ in $\Ga$ if and only if there exist $0 \le i < k$, $a \in \FFF_{q^2} \setminus \{1\}$, $b \in \FFF_{q^2}$ and $c \in \FFF^*_{q^2}$ with $b + b^q = a^{q+1}$, such that 
\begin{itemize}
\item[\rm (i)] $\bu_2 = (a_2, b_2, c_2)$ satisfies $a_2 = ac/(1-a), b_2 = bc/(1-a)$ and $c_2 = c/(1-a)$;
\item[\rm (ii)] $L_2$ is given by 
\begin{equation}
\label{eq:nb1}
N^i_{a, b, c}: \; (\l^{qp^{ir}} + a^q c)x - cy - (\l^{qp^{ir}} a + b^q c)z = 0.
\end{equation}
\end{itemize}
\item[\rm (b)] $(\la \bu_2 \ra, L_2) \in V(q)$ is adjacent to $(\infty, L^*)$ if and only if there exist $0 \le i < k$ and $a, b, c \in \FFF^*_{q^2}$ with $b + b^q = a^{q+1}$ such that 
\begin{itemize}
\item[\rm (i)] $\bu_2 = (a_2, b_2, c_2)$ satisfies $a_2 = ac, b_2 = bc$ and $c_2 = c$;
\item[\rm (ii)] $L_2$ is given by 
\begin{equation}
\label{eq:nb4}
M^i_{a, b, c}: \; (a^{q+1}c - \l^{qp^{ir}})x - acy + a(\l^{qp^{ir}} - b^q c)z = 0.
\end{equation}
\end{itemize}
\end{itemize}
\end{lemma}

\begin{proof}
(a) Denote $\bu_1 = (0, 1, 0)$. Then $(\la \bu_2 \ra, L_2)$ is adjacent to $(\infty, L)$ if and only if there exist an integer $0 \le i < k$ and a nonisotropic $\bu_0 = (a_0, b_0, c_0) \in V(3, q^2)$ orthogonal to both $\bu_1$ and $\bu_2$ such that $L$ and $L_2$ are given by (\ref{eq:l1}) and (\ref{eq:l2}) respectively. It is clear that (\ref{eq:l1}) gives $L: x=z$ if and only if $c_0 + c_2 = a_0 + a_2 \ne 0$. Since $\bu_0, \bu_1$ are orthogonal, we have $c_0 = 0$ and so $c_2 = a_0 + a_2 \ne 0$. Using this and the assumption that $\bu_0$ is nonisotropic, we obtain $a_0 \ne 0$. Since $\bu_0, \bu_2$ are orthogonal, we then have $b_0 = a_0 (a_2/(a_0 + a_2))^q$. Since $\bu_2$ is isotropic, we have $(a_0 + a_2)^q b_2 + (a_0 + a_2) b_2^q = a_2^{q+1}$. Setting $a = a_2/(a_0 + a_2)$, $b = b_2/(a_0 + a_2)$ and $c = a_0$, we have $a \in \FFF_{q^2} \setminus \{1\}$, $c \in \FFF^*_{q^2}$, $b + b^q = a^{q+1}$, $a_2 = ac/(1-a), b_2 = bc/(1-a)$ and $c_2 = c/(1-a)$. One can check that $L_2$ given by (\ref{eq:l2}) is exactly $N^i_{a, b, c}$ as shown in (\ref{eq:nb1}). Conversely, if these conditions are satisfied, then $(\la \bu_2 \ra, L_2)$ is adjacent to $(\infty, L)$.

(b) Let $\bu_1 = (0, 1, 0)$. Then $(\la \bu_2 \ra, L_2)$ is adjacent to $(\infty, L^*)$ if and only if there exist an integer $0 \le i < k$ and a nonisotropic $\bu_0 = (a_0, b_0, c_0) \in V(3, q^2)$ orthogonal to both $\bu_1$ and $\bu_2$ such that $L^*$ and $L_2$ are given by (\ref{eq:l1}) and (\ref{eq:l2}) respectively. Since $\bu_0$ and $\bu_1$ are orthogonal, we have $c_0 = 0$. Since $\bu_0$ is nonisotropic, we then have $a_0 \ne 0$. One can see that (\ref{eq:l1}) becomes 
$c_2 x - (a_0 + a_2)z = 0$, which gives $L^*$ if and only if $c_2 \ne 0$ and $a_0 = -a_2$. Since $\bu_0$ and $\bu_2$ are orthogonal, we have $-a_0 a_2^q + b_0 c_2^q = 0$ and so $b_0 = -a_2^{q+1}/c_2^q$. Since $\bu_2$ is isotropic, we have $-(a_2/c_2)^{q+1} + (b_2/c_2) + (b_2/c_2)^q = 0$. Set $a = a_2/c_2$, $b = b_2/c_2$ and $c=c_2$. Then $a, b, c \ne 0$, $b+b^q=a^{q+1}$, $\bu_2 = (ac, bc, c)$, and (\ref{eq:l2}) can be simplified to give (\ref{eq:nb4}). 
\qed
\end{proof}

It is known that every line of $U_H(q)$ through $0$ other than $L^*$ is of the form:
$$
N(\eta): y = \eta x,\;\,\mbox{where $\eta \in \FFF_{q^2}^*$}.
$$

\begin{lemma}
\label{lem:nbn}
\begin{itemize}
\item[\rm (a)] 
$(\la \bu_2 \ra, L_2) \in V(q)$ is adjacent to $(0, N(\eta))$ if and only if there exist $0 \le i < k$,
$f \in \FFF_{q^2} \setminus \{1\}, g \in \FFF_{q^2}$ and $h \in \FFF^*_{q^2}$ with $\eta^q g + \eta g^q = f^{q+1}$, such that
\begin{itemize}
\item[\rm (i)] $\bu_2 = (a_2, b_2, c_2)$ satisfies $a_2 = fh/(1-f), b_2 = \eta h/(1-f)$ and $c_2 = gh/(1-f)$;
\item[\rm (ii)] $L_2$ is given by 
\begin{equation}
\label{eq:nb3}
L(\eta)^i_{f,g,h}:\; \left(\l^{qp^{ir}} \eta^q + f^q h \right)x - \left(\l^{qp^{ir}} \eta^{q-1} f + g^q h\right)y - \eta^q h z = 0.
\end{equation}
\end{itemize}
\item[\rm (b)] $(\la \bu_2 \ra, L_2) \in V(q)$ is adjacent to $(0, L^*)$ if and only if there exist $0 \le i < k$ and $f, g, h \in \FFF^*_{q^2}$ with $g + g^q = f^{q+1}$ such that 
\begin{itemize}
\item[\rm (i)] $\bu_2 = (a_2, b_2, c_2)$ satisfies $a_2 = fh, b_2 = h$ and $c_2 = gh$;
\item[\rm (ii)] $L_2$ is given by 
\begin{equation}
\label{eq:nb5}
K^i_{f,g,h}:\; (\l^{qp^{ir}} - f^{q+1} h)x - f (\l^{qp^{ir}} - g^q h)y + fhz = 0.
\end{equation}
\end{itemize}
\end{itemize}
\end{lemma}

\begin{proof}
(a) Denote $\bu_1 = (0, 0, 1)$. Then $(\la \bu_2 \ra, L_2)$ is adjacent to $(0, N(\eta))$ if and only if there exist an integer $0 \le i < k$ and a nonisotropic $\bu_0 = (a_0, b_0, c_0) \in V(3, q^2)$ orthogonal to both $\bu_1$ and $\bu_2$ such that $N(\eta)$ and $L_2$ are given by (\ref{eq:l1}) and (\ref{eq:l2}) respectively. Since $\bu_0, \bu_1$ are orthogonal, we have $b_0 = 0$. Using this and the fact that $\bu_0$ is nonisotropic, we get $a_0 \ne 0$. One can see that (\ref{eq:l1}) becomes $b_2 x - (a_0 + a_2)y = 0$, which
gives $N(\eta)$ if and only if $a_0 + a_2 \ne 0$ and $b_2 = \eta (a_0 + a_2)$.  
Since $\bu_0, \bu_2$ are orthogonal, we have $-a_0 a_2^q + c_0 b_2^q = 0$ and hence $c_0 = a_0 (a_2/b_2)^q = a_0 (a_2/(a_0 + a_2))^q/\eta^q$. Since $\bu_2$ is isotropic, we have $\eta (a_0 + a_2) c_2^q + \eta^q (a_0 + a_2)^q c_2 = a_2^{q+1}$. Setting $f = a_2/(a_0 + a_2)$, $g = c_2/(a_0 + a_2)$ and $h = a_0$, we have $f \in \FFF_{q^2} \setminus \{1\}$, $h \in \FFF^*_{q^2}$, $\eta^q g + \eta g^q = f^{q+1}$, $a_2 = fh/(1-f)$, $b_2 = \eta h/(1-f)$, $c_2 = gh/(1-f)$, and $L_2$ given by (\ref{eq:l2}) is exactly $L(\eta)^i_{f, g, h}$ in (\ref{eq:nb3}). 

(b) Let $\bu_0$ and $\bu_1$ be as above. As in (a), we have $b_0 = 0$ and $a_0 \ne 0$. One can see that (\ref{eq:l1}) becomes $b_2 x - (a_0 + a_2)y = 0$, which gives $L^*$ if and only if $a_0 = -a_2$ and $b_2 \ne 0$. Since $\bu_0, \bu_2$ are orthogonal, we have $c_0 = a_0 (a_2/b_2)^q =  -a_2^{q+1}/b_2^q$.  Set $f = a_2/b_2$, $g = c_2/b_2$ and $h = b_2$. Then $f, g, h \in \FFF^*_{q^2}$ and $g + g^q = f^{q+1}$ since $\bu_2 = (fh, h, gh)$ is isotropic. Now $L_2$ given by (\ref{eq:l2}) is exactly $K^i_{f, g, h}$ in (\ref{eq:nb5}). 
\qed
\end{proof}

For $(\s, M) \in V(q)$, denote
$$
\Ga(\s, M) = \mbox{neighbourhood of $(\s, M)$ in $\Ga$}.
$$  
In other words, $\Ga(\s, M)$ is the set of vertices of $\Ga$ adjacent to $(\s, M)$.
Note that $L = L(\l^q)^0_{0, 0, \l^q}$, $N = N_{0, 0, 1}^0 = N(\l^q)$ and in general $N_{0,0,1}^i = N(\l^{qp^{ir}})$. Lemmas \ref{lem:nb}(a) and  \ref{lem:nbn}(a) imply:

\begin{corollary}
\label{cor:nbn}
We have
$$
\Ga(\infty, L) \cap B(0) = \{(0, N(\l^{qp^{ir}})): 0 \le i < k\}.
$$ 
$$
\Ga(0, N) \cap B(\infty) = \{(\infty, L(\l^q)^i_{0,0,\l^q}): 0 \le i < k\}.
$$
In particular, $(\infty, L)$ and $(0, N)$ are adjacent in $\Ga$. Moreover, for distinct $\s, \t \in X$, any vertex $(\s, M) \in B(\s)$ other than $(\s, L(\s\t))$ has exactly $k$ neighbours in $B(\t)$. 
\end{corollary} 
 
The last statement follows from the fact that $|\Ga(\infty, L) \cap B(0)|=k$, $G$ is 2-transitive on $X$, and $G_{\infty, 0}$ is transitive on $B(0) \setminus \{(0, L^*)\}$. Here and in the following $G_{\infty, 0}$ denotes the point-wise stabilizer of $\{\infty, 0\}$ in $G$, that is, the subgroup of $G$ consisting of those elements of $G$ which fix both $\infty$ and $0$. 

\bigskip
\begin{proof}{\bf of Theorem \ref{thm:con}}~~The statements in (a)-(b) can be restated as follows.
\begin{itemize}
\item[(a)] $|\Ga(\s, M) \cap \Ga(\t, K)| \ge q^2 (q-2)$, for any distinct $\s, \t \in X$ and any $(\s, M) \in B(\s), (\t, K) \in B(\t)$;
\item[(b)] $|\Ga(\s, N_1) \cap \Ga(\s, N_2)| \ge k(k-1)q$, for any $\s \in X$ and $(\s, N_1), (\s, N_2) \in B(\s)$ with $N_1 \ne N_2$.
\end{itemize}

\textit{Proof of (a):}~~Since $G$ is 2-transitive on $X$, it suffices to prove (a) for $\s = \infty$ and $\t = 0$.
Noting that $L(\infty 0) = L^*$, we have three possibilities to consider.

\smallskip 
Case 1: $M, K \ne L^*$.
\smallskip

Since $\Ga$ is $G$-arc transitive and $(\infty, L^*)$ is the only vertex of $B(\infty)$ not adjacent to any vertex of $B(0)$ (Theorem \ref{thm:ugraph}), $G_{\infty, 0}$ is transitive on $B(\infty) \setminus \{(\infty, L^*)\}$. So it suffices to prove $|\Ga(\infty, L) \cap \Ga(0, N(\eta))| \ge q^2 (q-2)$ for any $\eta \in \FFF_{q^2}^*$ in this case.

By Lemmas \ref{lem:nb}(a) and \ref{lem:nbn}(a), a vertex $(\la \bu_2 \ra, L_2) \in V(q)$ is adjacent to both $(\infty, L)$ and $(0, N(\eta))$ if and only if there exist $0 \le i, j < k$, $a, f \in \FFF_{q^2} \setminus \{1\}, b, g \in \FFF_{q^2}, c, h \in \FFF^*_{q^2}$ with $b+b^q = a^{q+1}$ and $\eta^q g+\eta g^q = f^{q+1}$ such that $a_2 = ac/(1-a) = fh/(1-f), b_2 = bc/(1-a) = \eta h/(1-f)$, $c_2 = c/(1-a) = gh/(1-f)$ and $L_2 = N_{a,b,c}^i = L(\eta)^j_{f,g,h}$. From these relations we have $f=\eta a/b$, $g = \eta/b$, $h = c(b-\eta a)/\eta (1-a)$. Thus the equation of $L(\eta)^j_{f,g,h}$ as given in (\ref{eq:nb3}) becomes 
$$
(\l^{qp^{jr}} b^q + a^q c d)x - (\l^{qp^{jr}} a b^{q-1} + cd)y - b^q cd z = 0,\;\mbox{where}\; d = (b-\eta a)/\eta (1-a).
$$
This equation gives $N_{a,b,c}^i$ (see (\ref{eq:nb1})) if and only if $(\l^{qp^{ir}} + a^q c)(\l^{qp^{jr}} a b^{q-1} + cd) = c (\l^{qp^{jr}} b^q + a^q c d)$ (which implies $b \ne 0$ as $c, h \ne 0$) and $(\l^{qp^{ir}} + a^q c) b^q cd = (\l^{qp^{ir}} a + b^q c)(\l^{qp^{jr}} b^q + a^q c d)$, or equivalently
\begin{equation}
\label{eq:c}
\left(\l^{q(p^{ir}-p^{jr})} d +b^{2q-1}\right)c = - \l^{qp^{ir}} ab^{q-1}.
\end{equation}
Since $b^q = a^{q+1}-b$, the coefficient of $c$ here is equal to zero if and only if $b$ satisfies a quadratic equation, which has at most two solutions. Since for any $0 \le i, j < k$ and $a \in \FFF_{q^2} \setminus \{1\}$, the equation $b+b^q = a^{q+1}$ about $b$ has $q > 2$ solutions, there are at least $q-2 \ge 1$ values of $b$ that satisfy $b+b^q = a^{q+1}$ and $\l^{q(p^{ir}-p^{jr})} d +b^{2q-1} \ne 0$. Each such tuple $(i, j, a, b)$ determines a unique $c$ via (\ref{eq:c}) and hence a unique common neighour of $(\infty, L)$ and $(0, N(\eta))$. Moreover, since $\la \bu_2 \ra = \la a, b, 1 \ra$, for different pairs $(a, b)$ the vertices $(\la \bu_2 \ra, L_2)$ belong to different blocks of $\BB$ and so are distinct. Therefore, $|\Ga(\infty, L) \cap \Ga(0, N(\eta))| \ge q^2 (q-2)$.   

\smallskip  
Case 2: $M = L^*$ but $K \ne L^*$.
\smallskip 

It suffices to prove $|\Ga(\infty, L^*) \cap \Ga(0, N(\eta))| \ge q^2 (q-2)$ for any $\eta \in \FFF_{q^2}^*$. By Lemmas \ref{lem:nb}(b) and \ref{lem:nbn}(a), a vertex $(\la \bu_2 \ra, L_2) \in V(q)$ is adjacent to both $(\infty, L^*)$ and $(0, N(\eta))$ if and only if there exist $0 \le i, j < k$, $a, b, c \in \FFF^*_{q^2}$ with $b+b^q = a^{q+1}$ and $f \in \FFF_{q^2} \setminus \{1\}, g \in \FFF_{q^2}, h \in \FFF^*_{q^2}$ with $\eta^q g+\eta g^q = f^{q+1}$ such that $a_2 = ac = fh/(1-f)$, $b_2 = bc = \eta h/(1-f)$, $c_2 = c = gh/(1-f)$ and $L_2 = M_{a,b,c}^i = L(\eta)_{f,g,h}^j$. From these relations we have $f=\eta a/b$ (which implies $\eta a \ne b$ as $f \ne 1$), $g = \eta/b$ and $h = c(b-\eta a)/\eta$. Plugging these into (\ref{eq:nb3}), the equation of $L(\eta)_{f,g,h}^j$ becomes
$$
(\l^{qp^{jr}} b^q + a^q c d)x - (\l^{qp^{jr}} ab^{q-1} + cd)y - b^q c d z = 0,\;\mbox{where}\; d = (b-\eta a)/\eta.
$$
This equation gives $M_{a,b,c}^i$ (see (\ref{eq:nb4})) if and only if $(a^{q+1}c - \l^{qp^{ir}})(\l^{qp^{jr}} ab^{q-1} + cd) = ac(\l^{qp^{jr}} b^q + a^q c d)$ and $- b^q c d (a^{q+1}c - \l^{qp^{ir}}) = a(\l^{qp^{ir}} - b^q c)(\l^{qp^{jr}} b^q + a^q c d)$, that is, 
$$
\left(ab^{2q-1} - \l^{q(p^{ir}-p^{jr})} d\right)c = \l^{qp^{ir}} ab^{q-1}.
$$
The remaining proof is similar to Case 1 above. 

\smallskip 
Case 3: $M = K = L^*$. 
\smallskip 

In this case we are required to prove $|\Ga(\infty, L^*) \cap \Ga(0, L^*)| \ge q^2 (q-2)$. By Lemmas \ref{lem:nb}(b) and \ref{lem:nbn}(b), a vertex $(\la \bu_2 \ra, L_2) \in V(q)$ is adjacent to both $(\infty, L^*)$ and $(0, L^*)$ if and only if there exist $0 \le i, j < k$, $a, b, c, f, g, h \in  \FFF^*_{q^2}$ with $b+b^q = a^{q+1}$ and $g+g^q = f^{q+1}$ such that $a_2 = ac = fh$, $b_2 = bc = h$, $c_2 = c = gh$ and $L_2 = M_{a,b,c}^i = K_{f,g,h}^j$. From these relations we have $f=a/b, g = 1/b$ and $h = bc$. Plugging these into (\ref{eq:nb5}), the equation of $K_{f,g,h}^j$ becomes
$$
(\l^{qp^{jr}} b^q - a^{q+1} c)x - a (\l^{qp^{jr}} b^{q-1} - c)y + ab^q cz = 0.
$$
This is identical to $M_{a,b,c}^i$ (see (\ref{eq:nb4})) if and only if $(a^{q+1}c - \l^{qp^{ir}}) a (\l^{qp^{jr}} b^{q-1} - c) = ac(\l^{qp^{jr}} b^q - a^{q+1} c)$ and $acab^q c = a(\l^{qp^{ir}} - b^q c)a (\l^{qp^{jr}} b^{q-1} - c)$, that is,
$$
\left(\l^{q(p^{ir}-p^{jr})} + b^{2q-1}\right) c = \l^{qp^{ir}} b^{q-1}.
$$
The rest of the proof is similar to Case 1 above.
 
\textit{Proof of (b):}~~Since $\Ga$ is $G$-vertex transitive, it suffices to prove $|\Ga(0, N_1) \cap \Ga(0, N_2)| = k(k-1)q$ for distinct $(0, N_1), (0, N_2) \in B(0)$. 

Consider $(0, N(\eta_1)), (0, N(\eta_2)) \in B(0) \setminus \{(0, L^*)\}$ first, where $\eta_1, \eta_2 \in \FFF_{q^2}^*$ are distinct. By Lemma \ref{lem:nbn}(a), a vertex $(\la \bu_2 \ra, L_2) \in V(q)$ is in both $\Ga(0, N(\eta_1))$ and $\Ga(0, N(\eta_2))$ if and only if there exist $0 \le i, j < k$,
$f_t \in \FFF_{q^2} \setminus \{1\}, g_t \in \FFF_{q^2}$ and $h_t \in \FFF^*_{q^2}$ with $\eta_t^q g_t + \eta_t g_t^q = f_t^{q+1}$ such that $\bu_2 = (a_2, b_2, c_2)$ satisfies $a_2 = f_t h_t/(1-f_t), b_2 = \eta_t h_t/(1-f_t)$ and $c_2 = g_t h_t/(1-f_t)$, for $t=1,2$, and $L_2 = L(\eta_1)^i_{f_1, g_1, h_1} = L(\eta_2)^j_{f_2, g_2, h_2}$. Thus $f_2 = (\eta_2/\eta_1)f_1$, $g_2 = (\eta_2/\eta_1)g_1$ and 
$h_2 = h_1 (\eta_1 -\eta_2 f_1)/(\eta_2 -\eta_2 f_1)$. Note that $f_2 \ne 1$ implies $f_1 \ne \eta_1/\eta_2$. Using these relations, the equation of $L(\eta_2)^j_{f_2, g_2, h_2}$ (see (\ref{eq:nb3})) can be simplified to 
$$
\left(\l^{qp^{jr}} \eta_1^q \cdot \frac{\eta_2 -\eta_2 f_1}{\eta_1 -\eta_2 f_1}+ f_1^q h_1 \right)x - \left(\l^{qp^{jr}} \eta_1^{q-1} f_1 \cdot \frac{\eta_2 -\eta_2 f_1}{\eta_1 -\eta_2 f_1} + g_1^q h_1\right)y - \eta_1^q h_1 z = 0.
$$  
This gives the equation of $L(\eta_1)^i_{f_1, g_1, h_1}$ (see (\ref{eq:nb3})) if and only if $(\eta_2 -\eta_2 f_1)/(\eta_1 -\eta_2 f_1) = \l^{q(p^{ir}-p^{jr})}$, or equivalently 
$$
f_1 = \frac{\l^{q(p^{ir}-p^{jr})} \eta_1 - \eta_2}{\l^{q(p^{ir}-p^{jr})} \eta_2 - \eta_2}. 
$$
Here we note that $\l^{q(p^{ir}-p^{jr})} \ne 1$ for $0 \le i \ne j < k$. Since $\eta_1 \ne \eta_2$, the right-hand side of this expression is neither $1$ nor $\eta_1/\eta_2$. Thus there are $k(k-1)$ possible choices of $f_1$, and each of them corresponds to exactly $q$ values of $g_1$ by $\eta_1^q g_1 + \eta_1 g_1^q = f_1^{q+1}$. It follows that $|\Ga(0, N(\eta_1)) \cap \Ga(0, N(\eta_2))| = k(k-1)q$.

It remains to prove $|\Ga(0, L^*) \cap \Ga(0, N(\eta))| = k(k-1)q$ for any $\eta \in \FFF^*_{q^2}$. By Lemma \ref{lem:nbn}, a vertex $(\la \bu_2 \ra, L_2) \in V(q)$ is in both $\Ga(0, L^*)$ and $\Ga(0, N(\eta))$ if and only if there exist $0 \le i, j < k$, $f_1, g_1, h_1 \in \FFF^*_{q^2}$ with $g_1 + g_1^q = f_1^{q+1}$, and $f_2 \in \FFF_{q^2} \setminus \{1\}, g_2 \in \FFF_{q^2}$, $h_2 \in \FFF^*_{q^2}$ with $\eta^q g_2 + \eta g_2^q = f_2^{q+1}$, such that $\bu_2 = (a_2, b_2, c_2)$ satisfies $a_2 = f_1 h_1 = f_2 h_2/(1-f_2), b_2 = h_1 = \eta h_2/(1-f_2)$ and $c_2 = g_1 h_1 = g_2 h_2/(1-f_2)$, and $L_2 = K^i_{f_1, g_1, h_1} = L(\eta)^j_{f_2, g_2, h_2}$. Thus $f_2 = \eta f_1$, $g_2 = \eta g_1$, $h_2 = h_1 (1 - \eta f_1)/\eta$, and so $f_1 \ne 1/\eta$ as $f_2 \ne 1$. Using these relations and (\ref{eq:nb3}), the equation of $L(\eta)^j_{f_2, g_2, h_2}$ can be simplified to 
$$
\left(-\l^{qp^{jr}} \cdot \frac{\eta f_1}{1-\eta f_1} - f_1^{q+1} h_1\right)x + f_1 \left(\l^{qp^{jr}} \cdot \frac{\eta f_1}{1-\eta f_1} + g_1^q h_1\right)y + f_1 h_1 z = 0.
$$
One can see that this gives $K^i_{f_1, g_1, h_1}$ (see (\ref{eq:nb5})) if and only if 
$$
f_1 = \frac{\l^{q(p^{ir}-p^{jr})}}{(\l^{q(p^{ir}-p^{jr})} - 1)\eta}. 
$$
Note that the right-hand side of this equation is neither $0$ nor $1/\eta$. Similarly as in the previous paragraph, we obtain $|\Ga(0, L^*) \cap \Ga(0, N(\eta))| = k(k-1)q$. 

So far we have completed the proof of (a) and (b).
 
Note that $\Ga$ is not a complete graph since, for example, $(\infty, L^*)$ and $(0, L^*)$ are not adjacent. Since $q > 2$, by (a) the distance in $\Ga$ between any two nonadjacent vertices is equal to two. So $\Ga$ has diameter two. Since $(0, N)$ and $(\infty, L)$ are adjacent by Corollary \ref{cor:nbn} and they have at least one common neighbour by (a), $\Ga$ has girth three. 
\qed
\end{proof}

\begin{proof}{\bf of Theorem \ref{thm:deg diam}}~~Let $q = p^e > 2$. Choose $r=e$ and $\l \in \FFF^*_q$. It is trivial that $\l^q$ ($=\l^{p^r}$) is in the $\la \psi^r \ra$-orbit containing $\l$. Hence $\Ga_{e,\l}(q)$ is well-defined, and is connected of diameter two by Theorem \ref{thm:con}. The assumption $\l \in \FFF^*_q$ ensures $\l^q = \l$ and so $k_{e,\l}(q) = 1$. Thus, by Theorem \ref{thm:ugraph}, $\Ga_{e,\l}(q)$ has order $q^2 (q^3 + 1)$ and degree $q(q^2 - 1)$. From this (\ref{eq:deg diam}) follows immediately. 
 
Now for $\De = q(q^2 - 1)$ we have $q > \De^{1/3}$. Thus $q^2 (q^3 + 1) = q^2 (\De + q + 1) = \De q^2 + q^3 + q^2 = \De (q^2 + 1) + q^2 + q > \De (\De^{2/3}+1) + \De^{2/3} + \De^{1/3}$ as claimed in (\ref{eq:deg diam1}).
\qed
\end{proof}

\section{Remarks}
\label{sec:rem}

In the case when $D=2$, the well known Moore bound \cite{MS} gives $N(\De, 2) \le \De^2 + 1$ for any $\De$. The equality holds only when $\De = 1, 2, 3, 7$ and possibly $57$, and for all other $\De$ we have $N(\De, 2) \le \De^2 - 1$ (see \cite{MS}). It is known \cite{Brown} that $N(\De, 2) \ge \De^2 - \De + 1$ for every $\De$ such that $\De - 1$ is a prime. It is proved in \cite{MMS} that the counterpart $N^{vt}(\De, 2)$ of $N^{at}(\De, 2)$ for vertex-transitive graphs satisfies $N^{vt}(\De, 2) \ge 8(\De + (1/2))^2/9$ if $\De = (3q-1)/2$, where $q$ is a prime power congruent to 1 modulo 4. This bound came with the discovery \cite{MMS} of an infinite family of vertex-transitive graphs $H_q$ (now well known as the McKay-Miller-\v{S}ir\'{a}\v{n} graphs) with degree $\De = (3q-1)/2$ and order $8(\De + (1/2))^2/9$. Since, as implied in \cite[Definition 11, Lemma 17]{Hafner}, such extremal graphs cannot be arc-transitive except for the Hoffman-Singleton graph $H_5$, the same bound may not apply to $N^{at}(\De, 2)$. 

In view of (\ref{eq:deg diam1}) and the comments above, it is natural to ask whether there exist infinitely many $\De \ge 3$ such that $N^{at}(\De, 2) \ge c \De^2$ for some constant $c > 0$. 
One may also ask whether there exists a constant $c > 0$ such that $N^{at}(\De, 2) \ge c \De^2$ for all $\De \ge 3$. However, this would not make much sense unless the same question for $N^{vt}(\De, 2)$ has an affirmative answer which, to the best of our knowledge, is unknown at present.

\bigskip

\noindent \textbf{Acknowledgements}\quad We appreciate Dr. Guillermo Pineda-Villavicencio for helpful discussions on the degree-diameter problem. The author was supported by a Future Fellowship (FT110100629) of the Australian Research Council.

\end{document}